\documentclass[a4paper,11pt]{amsart}

\usepackage{amssymb,amsmath,amsfonts,amsthm,mathabx,mathrsfs,pdfsync,
enumerate,bbm,fancyhdr,enumerate,graphicx, url,afterpage,setspace,geometry,mathabx}

\usepackage[dvipsnames]{xcolor}
\usepackage[colorlinks=true, pdfstartview=FitV, linkcolor=blue, citecolor=blue, urlcolor=blue]{hyperref}

\usepackage[english]{babel}
\usepackage[utf8]{inputenc}

\DeclareMathOperator{\Aut}{Aut}
\DeclareMathOperator{\C}{\mathbb{C}}

\DeclareMathOperator{\Z}{\mathbb{Z}}
\DeclareMathOperator{\N}{\mathbb{N}}
\DeclareMathOperator{\h}{\mathcal{H}}
\DeclareMathOperator{\hh}{\mathcal{\widehat{H}}}
\DeclareMathOperator{\B}{\mathcal{B}}
\DeclareMathOperator{\Tr}{Tr}

\newtheorem{thm}{Theorem}[section]
\newtheorem{lem}[thm]{Lemma}

\newtheorem{cor}[thm]{Corollary}

\theoremstyle{definition}
\newtheorem{defn}[thm]{Definition}
\newtheorem{example}[thm]{Example}

\theoremstyle{remark}
\newtheorem{rem}[thm]{Remark}

\numberwithin{equation}{section}

\begin{document}

\title{Von Neumann equivalence and $M_{d}$ type approximation properties}
\author[B. Battseren]{Bat-Od Battseren}
\address{Research Institute for Mathematical Sciences, Kyoto University, Kyoto 606-8502, JAPAN}
\curraddr{}
\email{batoddd@gmail.com}
\thanks{The author is supported by JSPS fellowship program (P21737).}
\subjclass[2020]{Primary 43A22; Secondary 16W22, 37A20, 20F65, 43A07, 46L07}

\date{}

\dedicatory{}
\begin{abstract}
We show that $M_d$-approximation-property, $M_d$-weak-amenability, and $M_d$-weak-Haagerup-property are stable under von Neumann equivalence (hence also Measure equivalence and W*-equivalence). We also show that these properties are inherited from lattices.
\end{abstract}

\maketitle

\section{Introduction}
Von Neumann equivalence (vNE) was introduced in \cite{IPR19} on the class of discrete countable groups as a coarser version of Gromov's Measure equivalence (ME) and W*-equivalence (W*E). In the same paper, the (non-)approximation type properties, amenability, Haagerup property (also known as a-T-menability), Kazhdan's property (T), and proper proximality, were shown to be stable under vNE. It is followed by some extension works that weak amenability, weak Haagerup property, Haagerup-Kraus's approximation property \cite{Ish21}, and exactness \cite{Bat22} are vNE invariants too. In this paper, we contribute by more invariants, namely $M_d$ type approximation properties. 

One of the motivations to study vNE is its similarity to ME. Many techniques used in measured group theory could be used in vNE, leading to potential applications in W*E. It is known that the groups $PSL_n(\Z)$, $n\geq 3$ represent distinct ME-classes \cite[Theorem A]{Fur99}, while it is an open problem for W*E. It would be delightful if the ME-rigidity result of \cite{Fur99} extends to vNE and W*E.

The algebras $M_d(G)$ were introduced in \cite{Pis05} by means of studying the similarity problem. Let $G$ be a locally compact group and let $d$ be an integer greater than 1. The space $M_d(G)$ consists of the functions $\varphi\in C_b(G)$ admitting Hilbert spaces $\h_i$ and bounded maps $\xi_i:G\rightarrow \B(\h_i,\h_{i-1})$ with $i=1,...,d$ and $\h_d = \h_0 = \C$ such that 
\begin{align*}
\varphi(x_1\cdots x_d) = \xi_1(x_1)\cdots \xi_d(x_d)(1)\quad \text{for all}\quad x_1,...,x_d\in G.
\end{align*}
With the norm $\|\varphi\|_{M_d} = \inf \sup_{x_1\in G}\|\xi_1(x_1)\|\cdots \sup_{x_d\in G}\|\xi_d(x_d)\|$ and pointwise product, $M_d(G)$ becomes a commutative Banach algebra. When $d=2$, this is the algebra $B_2(G)$ of Herz-Schur multipliers (also denoted by $M_0A(G)$ and $M_{cb}(G)$ in some literature). The similarity problem asks if all unitarizable groups are amenable (see \cite{Pis01} for a survey). Unitarizable groups indeed reflect amenable groups. For example, in \cite{Boz85} the amenability of $G$ was characterized by the equality $M_2(G) = B(G)$, and in \cite{Pis05} it is shown that if $G$ is unitarizable, then $M_d (G)= B(G)$ for some $d\geq 2$.

In \cite{Ver22}, the algebra $M_d(G)$ was used to define $M_d$-approximation-property (resp. $M_d$-weak-amenability), which is a refinement of Haagerup Kraus' approximation property (resp. weak amenability). When $d=2$ these properties coincide, and when $d$ increases the property gets stronger. As an application of this new approximation type property, it is shown in \cite[Theorem 1.2]{Ver22} that a discrete group satisfying $M_d$-approximation-property for all $d\geq 2$ has an affirmative answer for the similarity problem. 

Weak amenability and approximation property are known to be stable under vNE \cite{Ish21}. This fact inspired us to prove the following theorem.
\begin{thm}\label{thm A} Let $d$ be an integer greater than 1.
The following are stable under von Neumann equivalence.
\begin{enumerate}
\item $M_d$-approximation-property.
\item $M_d$-weak-amenability and $M_d$-Cowling-Haagerup constant $\Lambda_{CH}(G,d)$.
\item $M_d$-weak-Haagerup-property and $M_d$-weak-Haagerup constant $\Lambda_{WH}(G,d)$.
\end{enumerate}
\end{thm}
We call these three properties \textit{$M_d$ type approximation properties}. The precise definitions are given in Section \ref{section Md app prop}.
The key step of the proof lies in showing that if $\Gamma$ and $\Lambda$ are discrete countable vNE groups, then the induction map $\Phi: \ell^\infty(\Lambda)\rightarrow \ell^\infty (\Gamma)$ introduced in \cite{Ish21} gives a norm decreasing map $\Phi:M_d(\Lambda)\rightarrow M_d(\Gamma)$. As we mentioned before, vNE is coarser than both ME and W*E, hence the following.
\begin{cor} Let $d$ be an integer greater than 1.
$M_d$ type approximation properties, $M_d$-Cowling-Haagerup constant, and $M_d$-weak-Haagerup constant are stable under ME and W*E.
\end{cor}
In \cite[Theorem 1.3]{Ver22}, it was shown that countable groups acting properly on a finite-dimensional CAT(0) cube complex are $M_d$-weakly amenable. The proof was inspired by Valette's work involving analytic family of representations \cite{Val90}. We prove the following theorem using Bo\.{z}ejko-Picardello's method \cite{BP93}. Note that Theorem \ref{thm tree} is an immediate consequence of \cite[Theorem 1.3]{Ver22}, but the proof ideas differ.
\begin{thm}\label{thm tree}
A discrete group acting properly on a tree is $M_d$-weakly-amenable with $M_d$-Cowling-Haagerup constant 1 for any $d\geq 2$.
\end{thm}
Some group properties are extendable to von Neumann algebras in a way that a group has the property if and only if its group von Neumann algebra has the property. For example, a discrete group is  amenable if and only if its group von Neumann algebra is injective \cite{Con76}. This kind of characterization is usually observed for discrete groups. In this sense, working on discrete groups is advantageous. The following theorem allows one to work with discrete groups instead of its ambient group when investigating $M_d$ type approximation properties.
\begin{thm}\label{thm B}
Let $d$ be an integer greater than 1.
Let  $\Gamma$ be a lattice of a locally compact group $G$. If $\Gamma$ has $M_d$-approximation-property, then so does $G$. Moreover, we have $\Lambda_{CH}(G,d)\leq \Lambda_{CH}(\Gamma,d)$ and $\Lambda_{WH}(G,d)\leq \Lambda_{WH}(\Gamma,d)$.
\end{thm}
The assertions for $d=2$ were done in \cite[Theorem 2.4]{HK94}, \cite[Theorem 2.6]{Haa16}, and \cite[Theorem A]{Knu16}, respectively. 

\section{Preliminaries}
In this section, we recall some definitions and useful facts.
\subsection{$M_{d}$ type approximation properties}\label{section Md app prop}
Throughout the paper, $d$ is always an integer greater than 1, $G$ is a locally compact group, $H$ is a closed subgroup of $G$, and $\Gamma$ and $\Lambda$ are discrete groups. We denote by $G_{\mathrm{d}}$ the discrete realization of $G$. When $E$ is a set, we denote by $\chi_E$ the characteristic function of $E$.

The \textit{Fourier-Stieltjes algebra} $B(G)$ consists of the coefficients of all unitary representations of $G$. For $\varphi \in B(G)$, its norm is given by $\|\varphi\|_B = \inf \|u\|\|v\|$ where the infimum runs through all unitary representations $(\pi,\h)$ and vectors $u,v\in\h$ satisfying $\varphi (x) = \langle \pi (x)u,v\rangle$ for all $x\in G$. With this norm and pointwise product, $B(G)$ is a commutative Banach algebra. 
The \textit{Fourier algebra} $A(G)$  is a subalgebra of $B(G)$ consisting of all coefficients of the left regular representation $\lambda_G:G\rightarrow \mathcal{U} (L^2(G))$ endowed with the norm $\|\varphi\|_A = \inf \|u\|_2\|v\|_2$ where the infimum runs through all vectors $u,v\in L^2(G)$ satisfying $\varphi(x) = \langle \lambda_G(x)u,v\rangle$ for all $x\in G$. To mention some facts about the Fourier algebra: 
\begin{enumerate}
\item The infimum $\|\varphi\|_A = \inf \|u\|_2\|v\|_2$ reaches its minimum.
\item $A(G)$ is a commutative Banach algebra.
\item The inclusion $A(G)\rightarrow B(G)$ is isometric.
\item $C_c(G)\cap B(G)$ is dense in $A(G)$.
\item $A(G)\subseteq C_0(G)$.
\end{enumerate}
See \cite{Eym64},\cite{KL18}, and \cite[Appendix F]{Run20} for more about these algebras.

\begin{defn}[{\cite{Pis05}}]\label{defn Md} A continuous bounded function $\varphi\in C_b(G)$ is called a \textit{$M_d$-multiplier} 
if there exist Hilbert spaces $\h_i$ and bounded maps $\xi_i:G\rightarrow \B(\h_i,\h_{i-1})$ for $i=1,...,d$ such that $\h_d = \h_0 = \C$  and
\begin{align}\label{eq Md}
\varphi(g_1\cdots g_d) = \xi_1(g_1)\cdots \xi_d(g_d)\quad \text{for all}\quad g_1,...,g_d\in G.
\end{align} 
Here $\B(\C)$ is identified with $\C$. 
We denote by $M_d(G)$ the space of all $M_d$-multipliers of $G$ endowed with the norm $\|\varphi\|_{M_d} = \inf \sup_{g_1\in G}\|\xi_1(g_1)\|\cdots \sup_{g_d\in G}\|\xi_d(g_d)\|$ where the infimum runs through all $\h_i$'s and $\xi_i$'s satisfying \eqref{eq Md}. 
\end{defn}
When $d=2$, it is the algebra of completely bounded Fourier multipliers $M_{cb}A(G)$. 
It is obvious that the inclusion $M_{d+1}(G)\subseteq M_d(G)$ is norm decreasing. The coefficients of a unitary representation are typical elements of $M_d(G)$. To see that, take any $\varphi\in B(G)$. There exist a unitary representation $(\pi,\h)$ and vectors $u,v\in\h$ satisfying $\varphi (x) = \langle \pi (x)u,v\rangle$ for all $x\in G$. Then by choosing $\h_1 = ... = \h_{d-1} = \h$ and $\xi_d(g) (1)= \pi(g)u$, $\xi_i(g) = \pi(g)$, $\xi_1(g) (w)= \langle \pi(g) w, v\rangle$ for all $g\in G$, $w\in\h$, and $i=2,...,d-1$, we get  $\varphi(g_1\cdots g_d) = \langle \pi(g_1\cdots g_d)u,v\rangle = \xi_1(g_1)\cdots \xi_d(g_d)$ for all $g_1,...,g_d\in G$ and $\|\varphi\|_{M_d}\leq \|v\|\|w\|$. Thus we have the following norm decreasing inclusions
\begin{align*}
A(G)\subseteq B(G)\subseteq  M_{d+1} (G)\subseteq M_d (G)\subseteq M_{cb}A(G)\subseteq C_b(G).
\end{align*}

Note that the inclusion $M_d(G)\rightarrow M_d(G_{\mathrm{d}})$ is isometric. Indeed, we have $M_d(G) = M_d(G_{\mathrm{d}})\cap C(G)$. It is also worth noting that if $H$ is a closed subgroup of $G$, then the restriction map $M_d(G)\rightarrow M_d(H)$ is norm decreasing.

Just like $B(G)$ and $M_{cb} A(G)$, the space $M_d(G)$ is in duality with $L^1(G)$ via
\begin{align}\label{eq:duality}
\langle\varphi,\psi\rangle=\int_G \psi\varphi dx \quad  \text{for}\quad \varphi\in M_d(G),\psi\in L^1(G).
\end{align}
We denote by $X_d(G)$ the completion of $L^1(G)$ in the dual space $M_d(G)^*$. As explained in \cite{Pis05}, if $G$ is a discrete group, then for any $d\geq 2$ the Banach space $M_d(G)$ is isometrically isomorphic to the dual of $X_d(G)$. For $d=2$, the same result holds for general locally compact groups \cite[Proposition 1.10]{dCH85}. It seems that the case $d\geq 3$ for locally compact group is unclear. Nevertheless, for our use we only need the duality between $M_d(G)$ and $X_d(G)$ extended from \eqref{eq:duality}.
\begin{defn}\label{defn properties}
Let $G$ be a locally compact group and let $d\geq 2$ be an integer.
\begin{enumerate}
\item We say that $G$ has \textit{$M_d$-approximation-property} ($M_d$-AP in short) if there is a net $(\varphi_i)$ of functions in the Fourier algebra $A(G)$ such that $\varphi_i\rightarrow 1$ in $\sigma(M_d(G),X_d(G))$-topology.

\item We say that $G$ is \textit{$M_d$-weakly-amenable} ($M_d$-WA in short) if there is a net $(\varphi_i)$ of functions in $A(G)$ such that $\sup_i\|\varphi_i\|_{M_d}<\infty$ and $\varphi_i\rightarrow 1$ in $\sigma(M_d(G),X_d(G))$-topology. The minimum possible value of $\sup_i\|\varphi_i\|_{M_d}$ is called \textit{$M_d$-Cowling-Haagerup} constant and denoted by $\Lambda_{CH}(G,d)$. If $G$ is not $M_d$-WA, we simply write $\Lambda_{CH}(G,d)=\infty$.
\item We say that $G$ has \textit{$M_d$-weak-Haagerup-property} ($M_d$-WH in short) if there is a net $(\varphi_i)$ of functions in $M_d(G)\cap C_0(G)$ such that $\sup_i\|\varphi_i\|_{M_d}<\infty$ and $\varphi_i\rightarrow 1$ in $\sigma(M_d(G),X_d(G))$-topology. The minimum possible value of $\sup_i\|\varphi_i\|_{M_d}$ is called \textit{$M_d$-weak-Haagerup} constant and denoted by $\Lambda_{WH}(G,d)$. If $G$ does not have $M_d$-WH, we simply write $\Lambda_{WH}(G,d)=\infty$.
\end{enumerate}
\end{defn}
\begin{rem}
For $M_d$-AP and $M_d$-WA, since $C_c(G)\cap A(G)$ is dense in $A(G)$ and the inclusion $A(G)\subseteq M_d(G)$ is norm decreasing, we can assume that the functions $\varphi_i$ are compactly supported. It is also worth noting that for a bounded net $(\varphi_i)$ in $M_d(G)$, the uniform convergence on compact subsets is stronger than the convergence in $\sigma(M_d,X_d)$-topology.
\end{rem}

\begin{example}\label{example Z}
Let us check the simplest example, the infinite cyclic group $G=\Z$. Put $F_n = \{1,...,n\}$ and $\varphi_n = \langle \lambda(\cdot)\dfrac{\chi_{F_n}}{n^{1/2}},\dfrac{\chi_{F_n}}{n^{1/2}}\rangle \in A(\Z)$. Then 
\begin{align*}
\|\varphi_n\|_{M_d}\leq \|\varphi_n\|_{A} \leq \left \|\dfrac{\chi_{F_n}}{n^{1/2}}\right \|^2_2= 1,
\end{align*} and for any $m\in \Z$
\begin{align*}
\varphi_n(m) = \left\{ \begin{array}{ll}
\dfrac{n-|m|}{n}\quad &\text{if}\quad |m|<n\\
0\quad &\text{otherwise}
\end{array}\right. \longrightarrow 1
\end{align*} as $n\rightarrow\infty$. Therefore, $\Z$ is $M_d$-WA with $\Lambda_{CH}(\Z,d) = 1$. A similar proof works for any locally compact amenable groups.
\end{example}
We have the implications 
\begin{align*}
\text{Amenability}&\Rightarrow M_d\text{-WA}\Rightarrow M_d\text{-AP}\Rightarrow \text{AP},\\
M_d\text{-WA}&\Rightarrow M_d\text{-WH}, \quad M_{d+1}\text{-AP}\Rightarrow M_d\text{-AP},\\
M_{d+1}\text{-WA}&\Rightarrow M_d\text{-WA}, \quad M_{d+1}\text{-WH}\Rightarrow M_d\text{-WH}.
\end{align*}
Since the non-abelian free groups are $M_d$-WA \cite{Ver22}, the first implication is strict for any $d\geq 2$. The second implication is also strict for any $d\geq 2$ because of the semidirect product $SL_2(\Z)\ltimes \Z^2$ \cite[Corollary 4.4]{Ver22}. The fourth implication is also strict. Indeed, the wreath product $\Gamma =\Z\wr F_2$ satisfies Haagerup property \cite[Theorem 1.1]{CSV12}. In particular, it is $M_d$-WH with $\Lambda_{WH}(\Gamma,d)=1$ for all $d\geq 2$. However, it is not weakly amenable \cite[Corollary 4]{Oz12}.
The strictness of the other implications seem unclear.

We call the properties in Definition \ref{defn properties} \textit{$M_d$ type approximation properties}. When $d=2$, these properties and constants are extensively studied and provided many interesting examples and applications. We refer the readers to \cite{dCH85, Haa16, CH89, HK94, BO08, Knu16} for further references.
\begin{proof}[Proof of Theorem \ref{thm tree}] The proof idea is essentially from \cite{BP93}.
Suppose that $\Gamma$ acts properly on a tree $T$. We identify $T$ with its vertex set. Let $d$ be the usual distance on $T$. 

Fix a base point $v_0\in T$ and a geodesic ray $\gamma_0:\N_0\rightarrow T$ beginning from $v_0$. Every point $v\in T$ admits a geodesic ray $\gamma_v:\N_0\rightarrow T$ that begins at $v$ and eventually merges with $\gamma_0$. For $n\in \N_0$, define 
\begin{align*}
E_n = \{x\in \Gamma: d(xv_0,v_0) =n\},\quad  \chi_n = \chi_{E_n},\quad \text{and}
\quad \varphi_{n} =\sum_{i=0}^{\lfloor n/2\rfloor}\chi_{n-2i}.
\end{align*} Since the action is proper, these functions are finitely supported. Put $\h_0=\h_d=\C$ and 
\begin{align*}
\h_1 = ... = \h_{d-1} = \overline{\text{span}}\{\oplus_{k=0}^n\delta_{\gamma_{yv_0}(k)}:y\in \Gamma\} \subseteq \bigoplus_{k=0}^n \ell^2(T).
\end{align*} Define 
\begin{align*}
\xi_d(x)(1) &= \oplus_{k=0}^n \delta_{\gamma_{xv_0}(k)}\\
\xi_i(x)\left (\oplus_{k=0}^n\delta_{\gamma_{yv_0}(k)}\right ) &= \oplus_{k=0}^n\delta_{\gamma_{xyv_0}(k)}\\
\xi_1(x)\left (\oplus_{k=0}^n\delta_{\gamma_{yv_0}(k)}\right ) &= \langle\oplus_{k=0}^n\delta_{\gamma_{xyv_0}(k)},  \oplus_{k=0}^n \delta_{\gamma_{v_0}(n-k)} \rangle
\end{align*} for $x,y\in \Gamma$ and $i=2,...,d-1$. It is easily checked  that these are well defined bounded operators with $\|\xi_d(x)\|\leq (n+1)^{1/2}$, $\|\xi_i(x)\|=1$, and $\|\xi_1(x)\|\leq (n+1)^{1/2}$. Moreover, 
\begin{align*}
\xi_1(x_1)\cdots \xi_d(x_d)= \sum_{k=0}^n \langle\delta_{\gamma_{x_1\cdots x_dv_0}(k)}, \delta_{\gamma_{v_0}(n-k)} \rangle = \varphi_n(x_1\cdots x_d)
\end{align*} for all $x_1,...,x_d\in \Gamma$. It follows that $\|\varphi_n\|_{M_d}\leq n+1$ and $\|\chi_n\|_{M_d}=\|\varphi_n-\varphi_{n-2}\|_{M_d}\leq \|\varphi_n\|_{M_d}+\|\varphi_{n-2}\|_{M_d} \leq 2n$.

On the other hand, the kernels $(u,v)\in T\times T\mapsto \exp (-td(u,v))$ are positive definite for any $t>0$  \cite[Proposition 3.2]{BP93}. In particular, the functions $\rho_t: x\in \Gamma\mapsto \exp (-td(xv_0,v_0))$ are positive definite. Moreover, $\rho_t\rightarrow 1$ as $t\rightarrow 0$ and $\|\rho_t\|_{M_d}\leq \|\rho_t\|_{B}= \rho_t(e)= 1$. The only problem about the net $(\rho_t)$ is that we do not know if $\rho_t$ is in the Fourier algebra $A(\Gamma)$. To regulate that, we approximate $\rho_t$ by the finitely supported functions defined as
\begin{align*}
\varphi_{n,t} =  \rho_t \sum_{k=0}^n\chi_k=\sum_{k=0}^n \chi_k e^{-tk} \quad \text{for}\quad  n\in\N_0.
\end{align*}
Observe that
\begin{align*}
\|\rho_t-\varphi_{n,t}\|_{M_d} \leq \sum_{k=n+1}^\infty 2k e^{-tk}\rightarrow 0  
\end{align*}
as $n\rightarrow \infty$. This completes the proof.
\end{proof}

\subsection{Von Neumann equivalence (vNE)}
\begin{defn} Let $G$, $\Lambda$, and $\Gamma$ be discrete groups. Let $(M,\Tr)$ be a semifinite von Neumann algebra endowed with a faithful normal semifinite trace. The group of trace preserving *-automorphisms of $M$ is denoted by $\Aut(M,\Tr)$. A group homomorphism $\sigma: G\rightarrow \Aut(M,\Tr)$ is called a \textit{$G$-action on $(M,\Tr)$}. A \textit{fundamental domain} for the action $\sigma$ is a projection $p\in M$ such that $\sum_{x\in G}\sigma_x (p)= 1$, where the convergence is with respect to the strong operator topology. 
We say that $\Gamma$ is a \textit{von Neumann equivalence subgroup (or $vNE$-subgroup)} of $\Lambda$  if there exists an action
$\sigma: \Gamma\times \Lambda\rightarrow \Aut(M,\Tr)$ and fundamental domains $p$ and $q$ for each of $\Lambda$ and $\Gamma$ actions, respectively, such that the trace $\Tr(p)$ is finite. Furthermore, if the trace $\Tr(q)$ is finite, we say that $\Lambda$ and $\Gamma$ are von Neumann equivalent (vNE) and write $\Lambda\sim_{vNE} \Gamma$.
\end{defn}
\begin{example}
Suppose that $\Gamma$ is a subgroup of $\Lambda$. Then $\Gamma$ is a vNE-subgroup of $\Lambda$. To see that, define $M=\ell^\infty(\Lambda)$ and $\Tr (f)=\sum_{s\in \Lambda} f(s)$ for $f\in M_+$. Then we have the trace preserving action $\sigma: \Gamma\times \Lambda\rightarrow \Aut(M,\Tr)$, $\sigma_{\gamma,s}(f)(x)= f(s^{-1}x\gamma)$ for all $s,x\in\Lambda$, $\gamma\in\Gamma$, and $f\in M$. The projection $\delta_e \in M$ is a finite trace fundamental domain for $\Lambda$-action. If $S\subseteq \Lambda$ is a set of representatives of the left $\Gamma$-cosets, the projection $\chi_S\in M$ is a fundamental domain for $\Gamma$-action.
\end{example}
The vNE is originated from Measure equivalence (ME) and W*-equivalence (W*E). The following two examples show their connection.
\begin{example}
Recall that $\Gamma$ and $\Lambda$ are said ME if there is a standard measure space $(X,m)$ and commuting, measure preserving, free actions of $\Gamma$ and $\Lambda$ on $(X,m)$ with finite measure fundamental domains $\Omega_\Gamma$ and $\Omega_\Lambda$, respectively. In this case, $\Gamma\times \Lambda$ acts on $(M,\Tr)=(L^\infty(X,m),\int dm)$, and the characteristic functions $\chi_{\Omega_\Gamma},\chi_{\Omega_\Lambda}\in M$ are finite trace fundamental domains for $\Gamma$ and $\Lambda$ actions, respectively. Thus $\Lambda\sim_{vNE} \Gamma$.
\end{example}  
\begin{example}
Recall that $\Gamma$ and $\Lambda$ are W*E if their group von Neumann algebras $L(\Gamma)$ and $L(\Lambda)$ are *-isomorphic. In this case, by the standard representation argument, the Hilbert spaces $\ell^2(\Gamma)$, $L^2(L(\Gamma))$, $L^2(L(\Lambda))$, and $\ell^2(\Lambda)$ are isomorphic, and $\Gamma\times \Lambda$ acts on $M =\B(\ell^2(\Gamma))\cong\B(\ell^2(\Lambda))$ by left and right regular representations. The projection onto the subspace $\C \delta_e$ becomes a common fundamental domain with finite trace for both $\Gamma$ and $\Lambda$ actions, hence $\Lambda\sim_{vNE} \Gamma$.
\end{example} 
We need the following fact from \cite{IPR19,Ish21}.

\begin{lem}\label{lem theta}
Suppose that we have a trace preserving action $\sigma:G\rightarrow \Aut(M,\Tr)$. If $p\in M$ is a fundamental domain for $\sigma$, then the map $\theta_p:\ell^\infty (G)\rightarrow M$, $\theta_p(f) = \sum_{x\in G} f(x)\sigma_{x^{-1}} (p)$ is a normal faithful *-representation.
\end{lem}

\section{Main section}\label{section proof}
\subsection{Induction map}
In this subsection, we prove Theorem \ref{thm A}. The main tool is the induction map derived from vNE \cite[Lemma 3.1]{Ish21}. 
We explain the construction now. Suppose that we have an action $\sigma : \Gamma\times \Lambda \rightarrow \Aut (M,\Tr)$ that establishes $\Gamma$ as a vNE-subgroup of $\Lambda$. Let $p$ and $q$ be fundamental domains for $\Lambda$ and $\Gamma$ actions, respectively. We assume that the trace is normalized so that $\Tr (p)=1$. The induction map is defined as
\begin{align*}
\Phi: \varphi\in\ell^\infty(\Lambda) \mapsto \widehat{\varphi} \in \ell^\infty(\Gamma)
,\quad \widehat{\varphi}(\gamma) = \Tr (\sigma_\gamma(p)\theta_p (\varphi)) \quad \text{for all}\quad
\gamma\in\Gamma,
\end{align*} where $\theta_p$ is as in Lemma \ref{lem theta}.
In \cite[Lemma 3.1]{Ish21}, it was shown that the restriction $\Phi:M_2(\Lambda)\rightarrow M_2(\Gamma)$ of the above mapping is well defined and norm decreasing. The following lemma extends the latter result.
\begin{lem}\label{lem ind} Let $d\geq 2$ be an integer.
If $\varphi\in M_d(\Lambda)$, then $\widehat{\varphi}\in M_d(\Gamma)$ and $\|\widehat{\varphi}\|_{M_d}\leq \|\varphi\|_{M_d}$.
\end{lem}
\begin{proof}

Let $\varphi\in M_d(\Lambda)$. There exist Hilbert spaces $\h_1,...,\h_d$  and bounded maps $\xi_i:\Lambda \rightarrow \B (\h_i,\h_{i-1})$ for $i=1,...,d$ such that $\h_d=\h_0=\C$ and
\begin{align*}
\varphi(s_1\cdots s_d) = \xi_1(s_1)\cdots \xi_d(s_d)\quad \text{for all}\quad s_1,...,s_d\in \Lambda.
\end{align*}
We assume that $d$ is even so that we can write $d=2k$ for some integer $k$. The odd case is done similarly. 
Define the Hilbert spaces
\begin{align*}
\hh_0&=\hh_d = \C,
\\
\hh_{2i}  &= (L^2(M,\Tr)p) \overline{\otimes} \h_{2i} \quad \text{for} \quad i=1,...,k-1,
\\
\hh_{2i+1} &= (pL^2(M,\Tr)) \overline{\otimes} \h_{2i+1}\quad  \text{for}\quad  i=0,...,k-1.
\end{align*}
For $\gamma\in \Gamma$, define the operators $\widehat{\xi}_j(\gamma) : \hh_j\rightarrow \hh_{j-1}$ for $j=1,...,d$ as 
\begin{align*}
\widehat{\xi}_d (\gamma)(1) &=  \sum_{s\in \Lambda} p\sigma_{\gamma,s}(p) \otimes \xi_{d}(s)(1),
\\
\widehat{\xi}_{2i} (\gamma) (xp\otimes u) &=  \sum_{s\in \Lambda} p\sigma_{\gamma,s}(xp) \otimes \xi_{2i}(s)(u),
\\
\widehat{\xi}_{2i+1} (\gamma) (px\otimes v) &=  \sum_{s\in \Lambda} \sigma_{\gamma,s}(px)p \otimes \xi_{2i+1}(s)(v),
\\
\widehat{\xi}_1(\gamma) (px\otimes w) &=  \sum_{s\in \Lambda}\Tr\left ( \sigma_{\gamma,s}(px) p \right  ) \xi_1(s)(w)
\end{align*} for all $x\in L^2(M,\Tr)$, $u\in \h_{2i}$, $v\in \h_{2i+1}$, and $w\in \h_{1}$. We claim  that these are well defined bounded operators. Let us check $\widehat{\xi}_{2i}(\gamma)$. Take any $x=xp\in L^2(M,\Tr)p$ and $\gamma\in\Gamma$. Since $p$ is a finite trace fundamental domain for $\Lambda$-action, the vectors $\{p\sigma_{\gamma,s}(xp) :s\in \Lambda\}$ are pairwise orthogonal and summable in $pL^2(M,\Tr)$. The same thing happens for  $\{p\sigma_{\gamma,s}(x)\otimes \xi_{2i}(s)(v) :s\in \Lambda\}\subseteq \hh_{2i}$ because the operators $\{\xi_{2i}(s):s\in \Lambda \}$ are uniformly bounded. Now, take any finite sum 
$z=\sum_{j=1}^N x_j\otimes v_j\in \hh_{2i}$ where $x_j=x_jp\in L^2(M,\Tr)p$ are pairwise orthogonal. Then, noting that the elements $\{p\sigma_{\gamma,s}(x_j) \otimes \xi_{2i}(s)(v_j): s\in \Lambda, j=1,...,N\}$ are pairwise orthogonal, we get
\begin{align*}
\|\widehat{\xi}_{2i}(\gamma)(z) \|^2 &= \sum_{j=1}^N \sum_{s\in \Lambda} \|p\sigma_{\gamma,s}(x_j)\|^2 \|\xi_{2i}(s)(v_j)\|^2
\\
&\leq \sup_{s\in \Lambda}\|\xi_{2i}(s)\|^2 \sum_{j=1}^N \sum_{s\in \Lambda} \Tr(\sigma_s^{-1}(p)\sigma_\gamma(x_jx_j^*))\|v_j\|^2
\\
&= \sup_{s\in \Lambda}\|\xi_{2i}(s)\|^2 \sum_{j=1}^N \|x_j\|^2\|v_j\|^2
\\
&=\sup_{s\in \Lambda}\|\xi_{2i}(s)\|^2 \|z\|^2.
\end{align*} By density, it follows that $\widehat{\xi}_{2i}(\gamma)$ is a well defined bounded operator with $\|\widehat{\xi}_{2i}(\gamma)\|\leq \sup_{s\in \Lambda}\| \xi_{2i}(s)\|$. A similar argument works for $\widehat{\xi}_d(\gamma)$ and $\widehat{\xi}_{2i+1}(\gamma)$. The operator $\widehat{\xi}_1(\gamma)$ can be seen as the composition 
\begin{align*}
(pL^2(M,\Tr))\overline{\otimes} \h_1 \rightarrow (L^2(M,\Tr)p) \overline{\otimes} \h_0 \cong L^2(M,\Tr)p \xrightarrow{\Tr}
 \C,
\end{align*} where the first mapping defined by 
\begin{align*}
px\otimes w \mapsto \sum_{s\in \Lambda} \sigma_{\gamma,s}(px) p \otimes  \xi_1(s)(w)\quad \text{for all}\quad x\in L^2(M,\Tr)\quad \text{and}\quad  w\in \h_1
\end{align*} is bounded by the previous argument. The trace $\Tr$ is contracting when restricted to $L^2(M,\Tr)p$ because $\Tr (p)=1$. Thus we have $\|\widehat{\xi_1} (\gamma)\|\leq \sup_{s\in \Lambda}\|\xi_1(s)\|$.

Our next claim is that 
\begin{align*}
\widehat{\varphi} (\gamma_1\cdots \gamma_d) = 
\widehat{\xi}_1(\gamma_1)\cdots  \widehat{\xi}_d (\gamma_d)\quad \text{for all}\quad \gamma_1,...,\gamma_d\in \Gamma.
\end{align*}
Should this be true, we get $\widehat{\varphi}\in M_d(\Gamma)$ and $\|\widehat{\varphi}\|_{M_d}\leq \|\varphi\|_{M_d}$.
Observe that 
\begin{align*}
&\widehat{\xi}_1(\gamma_1)\cdots  \widehat{\xi}_d(\gamma_d)  = \widehat{\xi}_1(\gamma_1)\cdots  \widehat{\xi}_{d-1}(\gamma_{d-1}) \sum_{s_d\in \Lambda} p\sigma_{\gamma_d,s_d}(p) \otimes \xi_{d}(s_d)
\\
&= \widehat{\xi}_1(\gamma_1)\cdots  \widehat{\xi}_{d-2}(\gamma_{d-2})\sum_{s_d\in \Lambda}  \sum_{s_{d-1}\in \Lambda} \sigma_{\gamma_{d-1},s_{d-1}}(p)\sigma_{\gamma_{d-1}\gamma_d,s_{d-1}s_d}(p)p \otimes \xi_{d-1}(s_{d-1})\xi_{d}(s_d)
\\
&...
\\
&= \sum_{s_d\in \Lambda} ... \sum_{s_1\in \Lambda}  \Tr\left ( \Pi\left( \sigma_{\gamma_1\cdots \gamma_d,s_1\cdots s_d}(p), ...,\sigma_{\gamma_1,s_1}(p), p\right)\right ) \xi_{1}(s_1)\cdots \xi_{d}(s_d)
\\
&= \sum_{s_d\in \Lambda} ... \sum_{s_1\in \Lambda} \Tr\left (  \Pi\left( \sigma_{\gamma_1\cdots \gamma_d,s_1\cdots s_d}(p), ...,\sigma_{\gamma_1,s_1}(p), p\right)\right )  \varphi(s_1\cdots s_d)
\end{align*} for $\gamma_1,...,\gamma_d\in \Gamma$, where $\Pi:M^{d+1}\rightarrow M$ is the multiplication map defined by
\begin{align*}
\Pi(a_1,...,a_{d+1}) = a_d a_{d-2}\cdots  a_2 a_1 a_3 \cdots  a_{d+1}\quad \text{for all}\quad a_1,...,a_{d+1}\in M.
\end{align*}  By the change of variable $s_1 \mapsto s_1 s_d^{-1} \cdots  s_2^{-1}$
\begin{align*}
&\widehat{\xi}_1(\gamma_1)\cdots  \widehat{\xi}_d(\gamma_d) =
\\
&=\sum_{s_d\in \Lambda} ... \sum_{s_1\in \Lambda} \Tr\left ( \Pi\left( \sigma_{\gamma_1\cdots \gamma_d,s_1}(p), ...,\sigma_{\gamma_1,s_1s_d^{-1}\cdots s_2^{-1}}(p), p\right) \right ) \varphi(s_1)
\\
&=\sum_{s_d\in \Lambda} ... \sum_{s_2\in \Lambda} \Tr\left ( \Pi \left( \sigma_{\gamma_1\cdots \gamma_d}(p), ...,\sigma_{\gamma_1,s_d^{-1}\cdots s_2^{-1}}(p),  \left[\sum_{s_1\in\Lambda}\varphi(s_1)\sigma_{s_1^{-1}}(p)\right] \right) \right )
\\
&=\sum_{s_d\in \Lambda} ... \sum_{s_2\in \Lambda}\Tr\left (\Pi \left( \sigma_{\gamma_1\cdots \gamma_d}(p),... ,\sigma_{\gamma_1,s_d^{-1}\cdots s_2^{-1}}(p), \theta_p (\varphi) \right ) \right )
\\
&= \Tr\left ( \sigma_{\gamma_1\cdots \gamma_d}(p) \theta_p (\varphi) \right ) = \widehat{\varphi}( \gamma_1\cdots \gamma_d).
\end{align*} This completes the proof.
\end{proof}
\begin{lem}\label{lem Fourier algebra}
If $\varphi\in A(\Lambda)$, then $\widehat{\varphi}\in A(\Gamma)$ with $\|\widehat{\varphi}\|_{A(\Gamma)}\leq \|\varphi\|_{A(\Lambda)}$.
\end{lem}
\begin{proof}
This was explained in the proof of \cite[Theorem 1.1]{Ish21} using \cite[Proposition 4.2]{IPR19}. We reproduce the proof for the convenience of the readers. 

Recall that the Koopman representation associated to the action $\sigma:\Gamma\rightarrow \Aut(M,\Tr)$ is the unitary representation given by
\begin{align*}
\sigma^0_\Gamma: \Gamma\rightarrow \mathcal{U} (L^2(M,\Tr)), \quad \sigma^0_\Gamma(\gamma) x = \sigma_\gamma(x)\quad \text{for all}\quad \gamma\in \Gamma\quad\text{and}\quad x\in \mathfrak{n_{\Tr}},
\end{align*} where $\mathfrak{n}_{\Tr} = \{ x \in M : \Tr (x^*x)<\infty \}$ is dense in $L^2(M,\Tr)$. 

If $\varphi\in A(\Lambda)$, there exist vectors $\xi,\eta\in\ell^2(\Lambda)$ such that $\varphi (s) = \langle \lambda_\Lambda (s)\xi,\eta\rangle$ for all $s\in \Lambda$ and $\|\varphi\|_{A(\Lambda)} = \|\xi\|_2\|\eta\|_2$. Define the vectors
\begin{align*}
\hat{\xi} = \sum_{s\in\Lambda} \xi(s)\sigma_s^{-1} (p) \quad \text{and}\quad 
\hat{\eta} = \sum_{t\in\Lambda} \eta(t)\sigma_t^{-1} (p) \quad \text{in}\quad L^2(M,\Tr).
\end{align*} The calculation 
\begin{align*}
\widehat{\varphi}(\gamma) &= \Tr(\sigma_\gamma (p)\sum_{s\in \Lambda} \varphi(s)\sigma_s^{-1} (p))
\\
&=\sum_{t\in\Lambda}\sum_{s\in \Lambda}\xi(s^{-1}t)\overline{\eta(t)}\Tr(\sigma_\gamma (p) \sigma_s^{-1} (p))
\\
&=\sum_{t\in\Lambda}\sum_{s\in \Lambda}\xi(s)\overline{\eta(t)}\Tr(\sigma_\gamma (p) \sigma_{st^{-1}} (p))
\\
&=\sum_{t\in\Lambda}\sum_{s\in \Lambda}\xi(s)\overline{\eta(t)}\Tr(\sigma_{\gamma,s^{-1}} (p) \sigma_{t^{-1}} (p))
\\
&=\langle \sigma^0_\Gamma (\gamma) \hat{\xi} ,\hat{\eta}\rangle
\end{align*} shows that $\widehat{\varphi}\in B(\Gamma)$ and $\|\widehat{\varphi}\|_{B(\Gamma)} \leq \|\hat{\xi}\| \|\hat{\eta}\|= \|\xi\|_2\|\eta\|_2 =\|\varphi\|_{A(\Lambda)}$. It remains to show that $\widehat{\varphi}\in A(\Gamma)$.

We claim that all coefficients of the Koopman representation $\sigma^0_\Gamma$ are in $A(\Gamma)$. Let $q\in M$ be a fundamental domain for $\Gamma$-action on $M$. Denote by $M^\Gamma$ the von Neumann subalgebra of $M$ consisting of all $\Gamma$-fixed elements. The mapping $\tau: x\in M^\Gamma_+\mapsto \Tr(qx)$ defines a faithful normal semifinite trace on $M^\Gamma$. The operator 
\begin{align*}
\mathcal{F}_q : \ell^2 (\Gamma) \overline{\otimes}L^2(M^\Gamma,\tau)\rightarrow L^2(M,\Tr), \quad \mathcal{F}_q(\delta_\gamma\otimes x) = \sigma_\gamma(q)x
\end{align*} for all $\gamma\in \Gamma$ and $x\in \mathfrak{n}_\tau$ defines a unitary operator. See \cite[Proposition 4.2]{IPR19} for detail of these facts. Observe that 
\begin{align*} 
\sigma^0_\Gamma(\alpha)(\mathcal{F}_q(\delta_\gamma\otimes x))& = \sigma_{\alpha}(\sigma_\gamma(q)x)=\sigma_{\alpha\gamma}(q)x = \mathcal{F}_q (\delta_{\alpha\gamma}\otimes x)
\\
&=\mathcal{F}_q ((\lambda_\Gamma \otimes id)(\alpha) (\delta_\gamma\otimes x))
\end{align*}
for all $\alpha,\gamma\in \Gamma$ and $x\in\mathfrak{n}_\tau\subseteq M^\Gamma$. In other words, the unitary operator $\mathcal{F}_q$ intertwines the amplification $\lambda_\Gamma\otimes id$ of the left regular representation and the Koopman representation $\sigma^0_\Gamma$. It follows that every coefficient of $\sigma^0_\Gamma$ is a coefficient of $(\lambda_\Gamma \otimes id, \ell^2(\Gamma)\overline{\otimes} L^2(M^\Gamma,\tau))$. Let $\xi', \eta'\in \ell^2(\Gamma)\overline{\otimes} L^2(M^\Gamma,\tau)$ be any vectors, and consider the function $\psi(\gamma) = \langle (\lambda_\Gamma \otimes id)(\gamma) \xi',\eta' \rangle$.  Let $\{v_i:i\in I\}$ be an orthonormal basis for $L^2(M^\Gamma,\tau)$. Then we can write 
\begin{align*}
\xi' = \sum_{\alpha\in \Gamma,i\in I} a_{\alpha,i} \delta_\alpha \otimes v_i\quad \text{and}\quad \eta' = \sum_{\beta\in \Gamma,j\in I} b_{\beta,j} \delta_\beta \otimes v_j
\end{align*} for some numbers $a_{\alpha,i},b_{\beta,j}\in \C$ such that 
\begin{align*}
\|\xi'\|^2=\sum_{\alpha\in \Gamma, i\in I} |a_{\alpha,i}|^2\quad \text{and}\quad \|\eta'\|=\sum_{\beta\in \Gamma, j\in I} |b_{\beta,j}|^2.
\end{align*} Developing $\psi (\gamma) = \langle (\lambda_\Gamma\otimes id)(\gamma) \xi',\eta'\rangle$ further, we get
\begin{align*}
\psi(\gamma) &= \sum_{\alpha\in \Gamma, i\in I} \sum_{\beta\in \Gamma, j\in I} a_{\alpha,i}b_{\beta,j}\langle \lambda_\Gamma (\gamma) \delta_\alpha,\delta_\beta\rangle\langle v_i,v_j\rangle
\\
&= \sum_{i\in I} \langle \lambda_\Gamma (\gamma) (\sum_{\alpha\in \Gamma}a_{\alpha,i}\delta_\alpha),(\sum_{\beta\in \Gamma}b_{\beta,i}\delta_\beta)\rangle.
\end{align*}
Using the Cauchy–Schwarz inequality, one can see that the sum
\begin{align*}
\sum_{i\in I} \Big\|\sum_{\alpha\in \Gamma}a_{\alpha,i}\delta_\alpha\Big\|_2\Big\|\sum_{\beta\in \Gamma}b_{\beta,i}\delta_\beta\Big\|_2
\end{align*} is convergent and bounded by $\|\xi'\|_2\|\eta'\|_2$. Thus $\psi\in A(\Gamma)$. In particular, $\widehat{\varphi}\in A(\Gamma)$ and 
$\|\widehat{\varphi}\|_{A(\Gamma)}= \|\widehat{\varphi}\|_{B(\Gamma)}\leq \|\varphi\|_{A(\Lambda)}.$
\end{proof}

\begin{lem}\label{lem sigma conv}
The induction map $\Phi:M_d(\Lambda)\rightarrow M_d(\Gamma)$ is $\sigma(M_d,X_d)$-$\sigma(M_d,X_d)$-continuous.
\end{lem}
\begin{proof} The case $d=2$ was done in \cite[Proposition 3.2.(b)]{Ish21}. The same argument works to see the general case. 

On one hand, the induction map $\Phi: \ell^\infty(\Lambda) \rightarrow \ell^\infty (\Gamma)$ is normal. Indeed, if $(\varphi_i)$ is a bounded net of positive functions in $\ell^\infty(\Lambda)$ with $\varphi_\infty = \sup_i \varphi_i \in \ell^\infty (\Lambda)$, then for any positive $f\in \ell^1(\Gamma)$ we have
\begin{align*}
\langle \widehat{\varphi}_i ,f \rangle &= \sum_{\gamma\in\Gamma} \sum_{s\in \Lambda} f(\gamma)\varphi_i(s) \Tr(\sigma_\gamma (p)\sigma_s^{-1}(p))
\\
&= \sum_{s\in \Lambda} \varphi_i(s) \sum_{\gamma\in\Gamma} f(\gamma)\Tr(\sigma_\gamma (p)\sigma_s^{-1}(p))
\\
&= \langle\varphi_i , g\rangle
\rightarrow \langle\varphi_\infty , g\rangle = \langle \widehat{\varphi}_\infty ,f \rangle,
\end{align*} where we note that the function
$g: s\in \Lambda \mapsto \sum_{\gamma\in\Gamma} f(\gamma)\Tr(\sigma_\gamma (p)\sigma_s^{-1}(p))$ is in $\ell^1(\Gamma)$ since
\begin{align*}
\|g\|_1 = \sum_{\gamma\in\Gamma} \sum_{s\in \Lambda} f(\gamma)\Tr(\sigma_\gamma (p)\sigma_s^{-1}(p)) =\sum_{\gamma\in\Gamma} f(\gamma)=\|f\|_1.
\end{align*} Thus its dual map defines a contraction $\Phi^*|_{\ell^1(\Gamma)}: \ell^1(\Gamma)\rightarrow \ell^1(\Lambda)$. 

On the other hand, we have a contraction $\Phi|_{M_d(\Lambda)}:M_d(\Lambda)\rightarrow M_d(\Gamma)$ by Lemma \ref{lem ind}. It follows that the dual map $(\Phi |_{M_d(\Lambda)})^*: X_d(\Gamma) \rightarrow M_d(\Lambda)^*$ is also a contraction and coincides with $\Phi^*$  on $\ell^1(\Gamma)$. Thus by continuity and density argument, we have $\Phi^*: X_d(\Gamma)\rightarrow X_d(\Lambda)$ well defined and contracting. Now, if $\varphi_i\rightarrow \varphi_\infty$ in $\sigma(M_d(\Lambda),X_d(\Lambda))$, then for any $f\in X_d(\Gamma)$ we have
\begin{align*}
\langle \widehat{\varphi}_i,f \rangle = \langle \varphi_i,\Phi^*(f) \rangle \rightarrow \langle \varphi_\infty,\Phi^*(f) \rangle = \langle\widehat{\varphi}_\infty,f \rangle,
\end{align*} showing that the map $\Phi:M_d(\Lambda)\rightarrow M_d(\Gamma)$ is $\sigma(M_d,X_d)$-$\sigma(M_d,X_d)$-continuous.
\end{proof}

\begin{lem}[{\cite[Lemma 3.1.(b)]{Ish21}}]\label{lem c0}
If $\varphi\in c_0(\Lambda)$, then $\widehat{\varphi}\in c_0(\Gamma)$.
\end{lem}
This lemma was presented in \cite[Lemma 3.1.(b)]{Ish21} under the condition that the $\Gamma$-action on $(M,\Tr)$ is mixing, i.e. all coefficients of the Koopman representation $\sigma^0_\Gamma$ vanish at infinity. This condition follows from the assumption that the $\Gamma$-action admits a fundamental domain. Indeed, as we have seen in the proof of Lemma \ref{lem Fourier algebra}, the coefficients of the Koopman representation are in the Fourier algebra, hence vanish at infinity.
We are ready to prove Theorem \ref{thm A}.
\begin{proof}[Proof of Theorem \ref{thm A}]
Let $(\varphi_i)$ be a net of functions in $M_d(\Lambda)$ converging to 1 in $\sigma(M_d(\Lambda),X_d(\Lambda))$-topology. By Lemma \ref{lem ind} and Lemma \ref{lem sigma conv}, the net $(\widehat{\varphi}_i)$ is in $M_d(\Gamma)$ with $\sup_i\|\widehat{\varphi}_i\|_{M_d(\Gamma)} \leq \sup_i \|\varphi_i\|_{M_d(\Lambda)}$ and converges to 1 in $\sigma(M_d(\Gamma),X_d(\Gamma))$-topology. By Lemma \ref{lem Fourier algebra}, if the net $(\varphi_i)$ is in $A(\Lambda)$, then the net $(\widehat{\varphi}_i)$ is in $A(\Gamma)$. It follows that if $\Lambda$ has $M_d$-AP, then so does $\Gamma$, and that $\Lambda_{WA}(\Gamma,d)\leq \Lambda_{WA}(\Lambda,d)$. The converse statements are also true by symmetry. Similarly, Lemma \ref{lem c0} shows that $\Lambda_{WH} (\Gamma,d)=\Lambda_{WH} (\Lambda,d)$. This completes the proof.
\end{proof}
\begin{rem}
The above proof works to see that $M_d$ type approximation properties pass to vNE-subgroups.
\end{rem}
\begin{rem} One can give an alternative proof that all discrete countable amenable groups are $M_d$-WA for all $d\geq 2$ using 
Theorem \ref{thm A}. Indeed, if $\Gamma$ is a discrete countable amenable group, then $\Gamma$ is a ME-subgroup of $\Z$ \cite{OW80}, hence $\Gamma$ is $M_d$-WA by Theorem \ref{thm A} and Example \ref{example Z}.
\end{rem}

\subsection{Inheritance from lattices} 
In this section, we prove Theorem \ref{thm B}. Although this fact implicitly follows from the argument used in the previous subsection and some known similar results on weak amenability, approximation property, and weak Haagerup property \cite{Haa16,HK94, Knu16}, we provide a complete explanation for convenience.

Let us recall first how the case $d=2$ was handled. A discrete subgroup $\Gamma$ of $G$ is called a \textit{lattice} if there is a finite measure subset $\Omega\subseteq G$ such that $G = \bigsqcup_{\gamma\in \Gamma} \Omega\gamma$. In this case, there exist measurable maps $\gamma: G \rightarrow\Gamma$ and $\omega:G\rightarrow \Omega$ such that every element $g\in G$ can be uniquely written as $g= \omega(g)\gamma(g)$. The induction map used in \cite[Section 2]{Haa16} is defined as 
\begin{align*}
\Phi_1: \varphi\in M_2 (\Gamma)\mapsto \widetilde{\varphi}\in M_2(G),\quad \widetilde{\varphi}  = \chi_{\Omega}*(\varphi \mu_\Gamma)*\chi_{\Omega}^*,
\end{align*} where $\mu_\Gamma$ is the measure on $G$ that counts $\Gamma$-elements, and $\chi_\Omega$ is the characteristic functions of $\Omega$. In other terms, we can write
\begin{align*}
\widetilde{\varphi}(g)  = \int_{\Omega} \varphi(\gamma (gw))dw \quad \text{for all} \quad g\in G.
\end{align*} It is a crucial fact that $\widetilde{\varphi}$ is a continuous bounded function on $G$. Moreover, we have the following.
\begin{enumerate}
\item \cite[Lemma 2.1]{Haa16} If $\varphi\in A(\Gamma)$, then $\widetilde\varphi \in A(G)$.
\item \cite[Lemma 5.13]{Knu16} If $\varphi\in c_0 (\Gamma)$, then $\widetilde\varphi\in C_0(G)$.
\item \cite[Proof of Lemma 1.16]{HK94} The dual map $\Phi_1^*:M_2(G)^*\rightarrow M_2(\Gamma)^*$ restricts to a well defined map $\Phi_1^*:L^1(G)\rightarrow \ell^1(\Gamma)$. In particular, $\Phi_1^*: X_2(G)\rightarrow X_2(\Gamma)$ is well defined and norm decreasing, and $\Phi_1$ is $\sigma(M_2,X_2)$-$\sigma(M_2,X_2)$-continuous.
\end{enumerate}

On the other hand, putting $(M,\Tr) = (L^\infty (G),\int dx)$, we have the action $\sigma: G\times \Gamma\rightarrow \Aut(M,\Tr)$ given by $\sigma_{(g,\alpha)} f(x) = f(g^{-1}x\alpha)$ for $f\in M$, $g,x\in G$, and $\alpha\in \Gamma$. The projection $p=\chi_{\Omega}$ is a finite trace fundamental domain for $\Gamma$-action. This data is enough to construct the induction map 
\begin{align*}
\Phi:  M_d(\Gamma) \rightarrow  M_d(G_{\mathrm{d}}),\quad \widehat{\varphi}(g) = \Tr (\sigma_g(p)\theta_p (\varphi)) \quad \text{for all}\quad
g\in G,
\end{align*} where $G_{\mathrm{d}}$ is the discrete realization of $G$.
It turns out that $\Phi_1$ and $\Phi$ coincide. Indeed, 
\begin{align*}
\widehat{\varphi}(g) &= \Tr (\sigma_g(p)\theta_p (\varphi)) = \int_G \chi_{\Omega}(g^{-1}x)\sum_{\alpha\in \Gamma}\varphi(\alpha)\chi_{\Omega}(x\alpha^{-1})dx 
\\
&= \int_\Omega \sum_{\alpha\in \Gamma}\varphi(\alpha)\chi_{\Omega}(gx\alpha^{-1})dx  = \int_\Omega \varphi(\gamma(gw))dw = \widetilde{\varphi}(g)
\end{align*} for all $g\in G$.
It follows that the map $\Phi:M_d(\Gamma)\rightarrow M_d(G)$ is well defined and norm decreasing. The same proof as (3) works to see that $\Phi$ is $\sigma(M_d,X_d)$-$\sigma(M_d,X_d)$-continuous. Now, Theorem \ref{thm B} is immediate.
\subsection{Some questions}
We end the paper with naturally arising questions.
\begin{enumerate}
\item Is unitarizability invariant under vNE? Does unitarizability pass to vNE-subgroups?
\item Is it possible to define $M_d$ type approximation properties for $C^*$-algebras, operator spaces, and von Neumann algebras in a compatible way with the current definition for discrete groups?
\end{enumerate}
The first question follows from the similarity problem because amenability pass to vNE-subgroups. The second question is quite promising  firstly because if two groups have isomorphic group von Neumann algebras, then they share the same value for $M_d$ type approximation properties by Theorem \ref{thm A}, and secondly because the case $d=2$ is done in \cite{Haa16,HK94,Knu16}.

\section*{Acknowledgment} The author is supported by JSPS fellowship program (P21737). The author is sincerely grateful to Narutaka Ozawa for the hospitality at the Research Institute for Mathematical Sciences, Kyoto University, and for the fruitful discussions.

\bibliographystyle{amsalpha}
\bibliography{mybibfile.bib}
\end{document}